\documentclass[12pt]{amsart}
\usepackage{amscd}
\usepackage{amssymb}

\allowdisplaybreaks[1]

\newtheorem{thm}{Theorem}[section]
\newtheorem{rema}[thm]{Remark}
\newtheorem{cor}[thm]{Corollary}
\newtheorem{lem}[thm]{Lemma}
\newtheorem{prop}[thm]{Proposition}

\newtheorem{defn}[thm]{Definition}
\theoremstyle{definition}

\numberwithin{equation}{section}

\newcommand{\K}{\mathbb K}

\newcommand{\R}{\mathbb R}

\newcommand{\N}{\mathbb N}

\begin{document}

\date{}
\title
{Linear family of Lie brackets on the space of matrices
$Mat(n\times m,\K)$ and Ado's Theorem}
\author[B. Dali]{B\'echir Dali }
\address{Department of Mathematics\\
Faculty of  sciences of Bizerte\\
7021 Zarzouna, Bizerte, Tunisia} \email{bechir.dali@yahoo.fr}


\begin{abstract}
In this paper we classify a linear family of Lie brackets on the
space of rectangular matrices $Mat(n\times m,\K)$ and we give an analogue of the Ado's Theorem.
We give also a similar classification on the algebra of the square matrices $Mat(n, \K)$ and as a
consequence, we prove that we can't built a
faithful representation of  the $(2n+1)$-dimensional Heisenberg
Lie algebra $\mathfrak{H}_n$ in a vector space $V$ with $\dim
V\leq n+1$. Finally, we prove that in the case of the algebra of
square matrices $Mat(n,\K)$, the corresponding Lie algebras
structures are a contraction of the canonical Lie algebra
$\mathfrak{gl}(n,\K)$.
\end{abstract}

\maketitle{Mathematics Subject Classification} (2000): 17B05,
11C20.


\bigskip
\section{Introduction}

We begin by setting some notations which will be used throughout
the paper. Let $\K$ be a field with characteristic $p=0$,
$Mat(n\times m,\K)$ be the linear space of $n\times m$ rectangular
matrices with coefficients in $\K$ and $Mat(n,\K)$ is the
associative algebra of square matrices with coefficients in $\K$.

When dealing with the problem of representation of finite
dimensional Lie algebras two presentations leap into mind: a
presentation by matrices and a presentation by an array of
structure constants. In the first presentation the Lie algebra is
given by a finite set of matrices $\{A_1,\dots, A_n\}$ that form a
basis of the Lie algebra $\frak g$. If A and B are two elements of
the space spanned by the $A_i$, then their Lie product is defined
as $ [A,~B]=A.B-B.A $ (where the $.$ denotes the ordinary matrix
multiplication). The second approach is more abstract. The Lie
algebra is a (abstract) vector space over a field $\K$ with basis
$\{x_1,\dots, x_n\}$ and the Lie multiplication is determined by
$$
[x_i, x_j]=\sum_{k=1}^nc_{i,j}^k x_k.
$$
Here $(c^k_{i,j})\in \K^3$ is an array of $n^3$ structure
constants that determines the Lie multiplication completely. By
linear algebra it is seen that from a presentation of a Lie
algebra by matrices, we easily can obtain a presentation of it by
structure constants. Ado's theorem states that the opposite
direction is also possible.

\begin{thm}\rm{(Ado)} Every finite-dimensional Lie algebra $\frak g$
over a field of characteristic zero can be viewed as a Lie algebra
of square matrices under the commutator bracket. More precisely
$\frak g$ has a linear representation $\rho$ over $\K$, on a
finite dimensional vector space $V$, that is a faithful
representation, making $\frak g$ isomorphic to a subalgebra of the
endomorphisms of $V$.
\end{thm}
If a such representation of $\frak g$ is built on a vector space
$V$, we can request if we can built such  representation in a
lower dimensional vector space $V$.

In the other hand, let $\mathcal{A}$ be a finite-dimensional
associative algebra over a field $ \mathbb{K}$ of a characteristic
$p=0$, then $\mathcal{A}$ is canonically equipped with Lie algebra
structure given by the following bracket
$$
[u,~v]=u.v-v.u\quad u,~v\in\mathcal{A}.
$$
For instance, if $\mathcal{A}= End(V)$ the set of linear
transformations on $V$ where  $V$ is a finite dimensional vector
space over $\K$.  In equivalent way, if we consider $Mat(n,\K)$
the set of square matrices with coefficients in $\K$, then we get
a canonical Lie algebra structure with the above commutator
bracket, this algebra is denoted by $\mathfrak {gl}(n,\K)$.

However the structure of associative algebra over $\mathcal A$ is
not unique \cite{Y1}, indeed if we fix $w\in\mathcal{A}$ then we can define
the product $(u\circ v)_w=u.w.v$ and with respect $\mathcal A$ is
again an associative algebra. This induces a new Lie algebra
structure, defined by the bracket
$$
[u,~v]_w=u.w.v-v.w.u \quad \quad u,~v\in\mathcal{A}\quad \quad
(1).
$$
Thus we obtain a family of Lie brackets, labelled by the element
$w$. It is readily seen that we actually have a linear space of
Lie brackets, since the sum of two such brackets is also a Lie
bracket of the same type, and a natural question is to classify
these structures. For instance if $w=0$ then the Lie algebra
$(\mathcal A,[~,~]_0)$ is just the abelian Lie algebra $\K^{n^2}$.
The above construction can be applied to the algebra of square
matrices $Mat(n,\K)$. It can be applied also even if $u, v , w$
are not $n\times n$ matrices, since $(1)$ makes sense when $u,
v\in Mat(n\times m,\K)$-the linear space of $n\times m$ matrices
and $w\in Mat(m\times n,\K)$- the linear space of $m\times n$
matrices.

In this paper we will deal with the algebra of the endomorphisms
of a finite-dimensional vector space $V$ or in equivalent way, the
algebra of square matrices $Mat(n,\K)$ and we will deal also with
$Mat(n\times m,\K)$-the linear space of $n\times m$ matrices.

Our first goal in this paper is to give a complete
classification of these Lie algebra
structures, on $Mat(n,\K)$ (respectively on $Mat(n\times m,\K)$),
labelled by the Lie brackets $~[~,~]_J$ with
$J\in Mat(n,\K)$ (respectively in $Mat(m\times n,\K)$).

The paper is organized as follows. First we begin by recalling
some classical results of linear algebra and matrix properties.
Secondly, we classify the Lie algebra structures on $Mat(n,\K)$ and
$Mat(n\times m,\K)$. As a consequence of these classifications, we prove an analogue of
the Ado's Theorem, namely, we prove that any finite dimensional Lie algebra
$\mathfrak g$ can be viewed as a Lie sub-algebra $(Hom(\K^m,\K^n), [~,~]_j)$ for some
positive integers $n,m$ and some matrix $J\in Mat(m\times n,\K)$. we
prove also that there is no faithful linear representation of
the classical $(2n+1)$-Heisenberg Lie algebra $\mathfrak{H}_{2n+1}$ in
$\mathfrak{gl}(V)$ with $\dim V\leq n+1$. Finally we prove that in the
case of the algebra of square matrices, the Lie algebra structure
labelled by the family of brackets $([~,~]_J)$ constructed on $Mat(n,\K)$
are a contraction of the canonical structure $\mathfrak{gl}(n,\K)$.

\section{Preliminaries}

Let $n, m$ be integers in $ \mathbb{N}^*=\mathbb{N}\setminus\{
0\}$ and let $Mat(n\times m,\K)$ be the linear vector space of
$n\times m$ rectangular matrices. We denote its canonical basis
$(E_{i,j})_{1\leq i\leq n,1\leq j\leq m}$ with
$$
E_{i,j}=(\delta_{p,i}\delta_{q,j})_{1\leq p\leq n,1\leq q\leq m},
$$
where $\delta_{p,i}$ is the Kronecker symbol.

\begin{defn}
Let $J$ and $J'$ be in $Mat(m\times n,\K)$. We say that $J$ is
equivalent to $J'$ if and only if there exist two invertible
matrices $P\in GL(n,\K)$ and $Q\in GL(m,\K)$ such that
$$
J=QJ'P.
$$
\end{defn}

\begin{rema}
(1) The  matrices $Q$ and $P$ are not unique.

(2) The above property of equivalent matrices in $Mat(m\times
n,\K)$ is an equivalence relation.

(3) $J$ and $J'$ are equivalent if and only if they represent the same
endomorphism in different bases.
\end{rema}

The  characterization of two equivalent matrices is given by the
 following

\begin{lem}
Let $J, J'\in Mat(m\times n, \K)$. $J$ and $J'$ are equivalent if
and only if they have the same rank.
\end{lem}

From the above Lemma, we can conclude the following

\begin{lem}
Let $J$ a matrix in $Mat(m\times n,\K)$ of rank $r\leq min (m,n)$,
then there exist two invertible matrices  $Q\in GL(m,\K)$ and
$P\in GL(n,\K)$ such that
$$
J=Q\begin{pmatrix}
I_r&0_{r,n-r}\\
0_{m-r,r}&0_{m-r,n-r}
\end{pmatrix}P,
$$
where $0_{p,q}$ is
the zero matrix of $Mat(p\times q,\K)$ and $I_r$ is the identity matrix of $Mat(r,\K)$.
\end{lem}

We will denote the matrix $\displaystyle{\begin{pmatrix}
I_r&0_{r,n-r}\\
0_{m-r,r}&0_{m-r,n-r}
\end{pmatrix}}$ by $J_{m,n,r}$
and when $n=m$ we simply denote it by $J_{n,r}$. Note that
$J_{n,n}=I_n$ the identity matrix.

Consider now the linear space $Mat(n\times m,\K)$ and for $J\in
Mat(m\times n,\mathbb{K})$ put
$$
[A,B]_J=AJB-BJA,\quad A,B\in Mat(n\times m,\mathbb{K}),
$$ then we have

\begin{lem}
(i)$ [A,B]_{J+\alpha J'}=[A,B]_J+\alpha[A,B]_{J'}, \forall
\alpha\in \mathbb{K}, J~,J'\in Mat(n\times m,\mathbb{K}).$

(ii) $\left(Mat(n\times m,\mathbb{K},[~,~]_J\right)$ is a Lie algebra.
\end{lem}
The proof is a simple verification of the anti-symmetry  of the
bracket $[~,~]_J$ and the Jacobi identity.

\noindent
Now we give the following
\begin{prop}\label{Equivalent matrices} If $J$ and $J'$ are equivalent matrices in
$Mat(m\times n, \mathbb K)$, then the corresponding Lie algebra structures
on $Mat(n\times m,\mathbb{K})$ are isomorphic.
\end{prop}

\begin{proof}
Since $J$ and $J'$ are equivalent in $Mat(m\times n, \K)$,
then there exist two invertible
matrices $Q\in GL(m,\K)$ and $P\in GL(n,\K)$ such that $ J=QJ'P.$ Consider now the map
$$
\begin{array}{lll}
\varphi:& Mat(\times m,\mathbb{K})&\longrightarrow Mat(n\times m,\mathbb{K})\\
       &A   &  \longmapsto PAQ\\\end{array}%
$$
we verify, since $Q$ and $P$ are invertible, that $\varphi$ is an
isomorphism of vector space and
$$
\aligned
\varphi\left([ A,B]_{J}\right)&=P(AJB-BJA)Q \\
                              &=P(AQJ'PB-BQJ'PA)Q\\
                              &=[PAQ,PBQ]_{J'}\\
                              &=[\varphi(A),\varphi(B)]_{J'}
\endaligned
$$
and thus $\varphi$ is a Lie algebra isomorphism.
\end{proof}
Let $J\in Mat(m\times n,\K)$ with $rk(J)=r$, from Proposition
\ref{Equivalent matrices}, $J$ is equivalent to the fllowing matrix
$\displaystyle{\begin{pmatrix}
I_r&0_{r,n-r}\\
0_{m-r,r}&0_{m-r,n-r}
\end{pmatrix}}$, we shall denote the last matrix by $J_{m,n,r}$. In the case
when $m=n$, we simply denote it by $J_{n,r}$. Recall that $J_{n,n}=I_n$, the
identity matrix of $Mat(n,\mathbb K)$.

\noindent For easy of notations, we shall denote the Lie algebra
$\bigl(Mat(n\times m,\mathbb K),[~,~]_{J_{m,n,r}}\bigr)$ by
$\mathfrak{gl}(n,m,r,\mathbb K)$ and in the case when $m=n$, we simply
denotes the lie algebra $\bigl(Mat(n,\mathbb K),[~,~]_{J_{n,r}}\bigr)$
by $\mathfrak{gl}(n,r,\mathbb K)$.
In order to give a complete classification of these Lie algebras structures labelled by the
linear family of Lie brackets $([~.~]_J)$, we first give the following
\begin{lem}\label{bracket}
Let $M\in Mat(n\times m,\K)$ and put $ M=\displaystyle{\begin{pmatrix}
M_1&M_3\\
M_2&M_4
\end{pmatrix}}$
with $ M_1\in Mat(r,\mathbb{K})$, $M_2\in Mat((n-r)\times
r,\mathbb{K}), M_3\in Mat(r\times(m-r),\mathbb{K}), M_4\in
Mat((n-r)\times (m-r),\mathbb{K})$. Then the Lie bracket $[A,B]_{J_{m,n,r}}$ is
given by
$$
[A,B]_{J_{n,r}}=\begin{pmatrix}[A_1,B_1]&A_1B_3-B_1A_3&\\
A_2B_1-B_2A_1&A_2B_3-B_2A_3)\end{pmatrix}.
$$
\end{lem}

\begin{proof}The proof is based on the following result. Put
$ A=\displaystyle{\begin{pmatrix}
A_1&A_3\\
A_2&A_4
\end{pmatrix}},$
with $ A_1\in Mat(r,\mathbb{K})$, $A_2\in Mat((n-r)\times
r,\mathbb{K}),A_3\in Mat((m-r)\times r,\mathbb{K}), A_4\in
Mat((n-r)\times (m-r),\mathbb{K})$. Then we have
$$
AJ_{m,n,r}=\begin{pmatrix}
A_1&0\\
A_2&0
\end{pmatrix}\hbox { and }
J_{m,n,r}A=\begin{pmatrix}
A_1&A_3\\
0&0
\end{pmatrix}.
$$
\end{proof}

In order to give a classification of the Lie algebras structures
on $Mat(n\times m,\mathbb K)$ labelled by the family of Lie brackets
$\bigl([~,~]_J\bigr)_J$, we shall distinguish the two possible cases
$n=m$ and $n\neq m$.

\section{The case of $Mat(n,\K)$}\label{square}

In this section we will deal with the linear vector space of
square matrices $Mat(n,\K)$. It is well known that it has a
structure of associative algebra and thus a Lie algebra with the
commutator
$$
[A,B]=AB-BA, \quad A, B\in Mat(n,\K).
$$
The commutator of two matrices $E_{i,j}$ and $E_{k,\ell}$ of the
canonical basis ( $1\leq i,j,k,l\leq n$) is given by
$$
[E_{i,j},E_{k,\ell}]=\delta_{j,k}E_{i,\ell}-\delta_{\ell,j}E_{k,j}.
$$
where $\delta$ denotes the Kronecker symbol.

Fix $J\in Mat(n,\K)$ and consider the bracket
$$
[A,B]_J=AJB-BJA, \qquad A, B\in Mat(n,\K).
$$
We first give the following

\begin{lem} Let $\mathfrak h$ be a Lie sub-algebra of $\mathfrak{gl}(n,\K)$, then
$\mathfrak h$ is a Lie sub-algebra of $\mathfrak{gl}(p,n,\K)$ for any integer $p\geq n$.
\end{lem}
\begin{proof}
For any subset $\mathfrak h\subset Mat(n,\K)$, and $p\geq n$, we have
$$
[\mathfrak h,\mathfrak h]_{J_{p,n}}=[\mathfrak h,\mathfrak h].
$$
This completes the proof.
\end{proof}

In order to classify the Lie algebras structures on $Mat (n,\K)$
labelled by the family of brackets $[~,~]_J$, we consider a vector space
$V$ over $\K$ with $\dim V=n$ and put
$V=V_1\oplus V_2$. Next, define $\mathfrak n$ to be the two step nilpotent
Lie algebra constructed on the vector space $Hom(V_1,V_2)\oplus
Hom(V_2,V_1)\oplus Hom(V_2,V_2)$ equipped with the bracket
$$
[(A,B,C),(A',B',C')]=(0,0,AB'-A'B).
$$
Finally, we define the Lie algebra $S(V_1,V_2)$ to be the semi direct
product of $End(V_1)$ with $\mathfrak n$. Then we have the
following

\begin{prop}\label{split}
\noindent (i) The Lie algebra $S(V_1,V_2)$ is isomorphic to
$\mathfrak{gl}(n,r,\K)$.

\noindent (ii) $End(V_1)$ is a Lie subalgebra of $\mathfrak{gl}(n,s,\K)$
for any $s\geq r$.

\noindent (iii) The reductive part of $\mathfrak{gl}(n,r,\K)$ is $End(V_1)\oplus End (V_2)$.
\end{prop}

\begin{proof}

\noindent (i) First, we can simply verify that $\mathfrak n$ is a two-step
nilpotent Lie algebra. Now for each $X\in End(V_1)$, consider
$$
\aligned
\pi_X:\  \mathfrak n &\longrightarrow \mathfrak n \\
           (A,B,C)& \longmapsto  (-AX,XB,0).
\endaligned
$$
Then we easily verify that $ X\longmapsto \pi_X$ is a Lie algebra
homomorphism from $\mathfrak{gl}(n,\K)$ into ${\partial}
(\mathfrak n)$ the Lie algebra of the derivations of $\mathfrak
n$.

\noindent Now consider the mapping
$$
\aligned
\phi:\  S(V_1,V_2) &\longrightarrow \mathfrak{gl}(n,r,\K) \\
           (X,A,B,C)& \longmapsto  \begin{pmatrix}
X&B\\
A&C
\end{pmatrix}.
\endaligned
$$
$\phi$ is a vector space isomorphism.

\noindent Let $X, X'\in End(V_1)$ and $(A,B,C),(A',B',C')\in
Hom(V_1,V_2)\oplus
Hom(V_2,V_1)\oplus Hom(V_2,V_2)$. Put $n=(A,B,C), n'=(A',B',C')$ then
$$
\begin{aligned}
\phi([(X,n), (X',n')])&=\phi(([X,X'],[n,n']+\pi_{X'}(n)-\pi_X(n')))\\
                &=[\phi((X,n)), \phi((X',n'))]_{J_{n,r}}.
\end{aligned}
$$

\noindent (ii) The statement is a simple consequence of Lemma \ref{bracket}.

\noindent (iii) It is easy to verify that $End(V_1)$ is reductive in $\mathfrak{gl}(n,r,\K)$.
Now, let $\mathfrak g$ be a reductive subalgebra in $\mathfrak{gl}(n,r,\K)$ which is maximal (for the inclusion).
From \cite{IT} we can write $\mathfrak g= \mathfrak z\oplus [\mathfrak g,\mathfrak g]$, where $\mathfrak z$ is the center
of $\mathfrak g$.
Since $End(V_1)\subset \mathfrak g$ then we must have
$[Z, X]=0,$ for any $Z\in\mathfrak z$ and $X\in End(V_1)$.

\noindent Put $Z=\displaystyle{\begin{pmatrix}
Z_1&Z_3\\
Z_2&Z_4
\end{pmatrix},  X=\begin{pmatrix}
X&0\\
0&0
\end{pmatrix}}$ then
$$
[Z,X]_{J_{n,r}}=\begin{pmatrix}
[Z_1,X]&-XZ_3\\
Z_2X&0
\end{pmatrix}=0,
$$
and thus
$Z=\begin{pmatrix}
\lambda I_r&0\\
02&Z_4
\end{pmatrix}$ with $\lambda\in \K$ and $Z_4\in End(V_2)$.

\noindent Finally, let
$Z=\begin{pmatrix}
\lambda I_r&0\\
0&Z_4
\end{pmatrix}\in \mathfrak z$ and $A=\begin{pmatrix}
A_1&A_3\\
A2&A_4
\end{pmatrix}\in\mathfrak g$, then we have
$$
[Z,A]_{J_{n,r}}=\begin{pmatrix}
0&\lambda A_3\\
-\lambda A_2&0
\end{pmatrix}=0.
$$
Thus $\mathfrak g=\begin{pmatrix}
End(V_1)&0_{r,n-r}\\
0_{n-r,r}&End(V_2)
\end{pmatrix}\simeq End(V_1)\oplus End(V_2),$ which completes the proof.
\end{proof}

Now we shall determine the center of the Lie algebra $(Mat(n,\K),[~,~]_J)$.
Let us denote it by $ \mathcal{Z}_{J}$, then we have the following

\begin{prop}\label{center}
Put $r=rank(J)$, if $r<n$ then $\dim \mathcal{Z}_{J}=(n-r)^2$
while if $r=n$, then $\dim \mathcal{Z}_{J}=1$.
\end{prop}

\begin{proof}

\noindent {\bf Case 1}: $rank(J)=n$

In this case, the mapping $\varphi: A\longmapsto JA$ is a Lie
algebra isomorphism from $\left(Mat(n,\mathbb{K}),[~,~]_J\right)$ onto
$\mathfrak {gl}(n,\mathbb{K})$, and since the center $\mathcal{Z}$ of
$\mathfrak {gl}(n,\mathbb{K})$ is $ \mathcal{Z}=\mathbb{K}I_n$ then $
\mathcal{Z}_{J}=\varphi^{-1}(\mathcal{Z})=\mathbb{K}J^{-1}.$

\noindent {\bf Case 2}: $rank(J)=r<n$

Let $Q, P$ in $GL(n,\mathbb{K})$ such that
$$
J=QJ_{n,r}P $$ and let $\varphi$ be the isomorphism mapping defined in
Proposition \ref{Equivalent matrices}, then
$$
\mathcal{Z}_{J}=\varphi^{-1}(\mathcal{Z}_{J_{n,r}}), $$ then we
shall compute $\mathcal{Z}_{J_{n,r}}$ the center of
$\mathfrak{gl}(n,r,\K)$.

$$
\mathcal{Z}_{J_{n,r}}=\{ A\in Mat(n,\mathbb{K}), |
~~[A,B]_{J_{n,r}}=0~\forall B\in Mat(n,\mathbb{K})\}.
$$
Put $\displaystyle{
A=\begin{pmatrix}
A_1&A_3\\
A_2&A_4
\end{pmatrix}},$ with $ A_1\in Mat(r,\mathbb{K}), A_2\in Mat((n-r)\times
r,\mathbb{K}),A_3\in Mat(r\times(n-r),\mathbb{K}), A_4\in
Mat(n-r,\mathbb{K})$, using Lemma \ref{bracket}, we get
$$
[A,B]_{J_{n,r}}=\begin{pmatrix}[A_1,B_1]&A_1B_3-B_1A_3&\\
A_2B_1-B_2A_1&A_2B_3-B_2A_3)\end{pmatrix}.
$$
Then $A\in\mathcal{Z}_{J_{n,r}}$ if and only if the bracket $
[A,B]_{J_{n,r}}$ is identically zero for any matrix $B$ in $Mat(n,\K)$ which
implies in particular $[A_1,B_1]=0,~~\forall
~~B_1\in Mat(r,\mathbb{K})$ and  thus $A_1$ is in the center of
$\mathfrak {gl}(r,\K)$, that is
$$
A_1=\lambda I_r,\qquad \lambda\in\K.
$$
Now with the equations
$$
A_1B_3-B_1A_3= A_2B_1-B_2A_1 =A_2B_3-B_2A_3=0,~~\forall ~~B~~\in Mat(n,\mathbb{K}),
$$
we obtain
$$
A=\begin{pmatrix}
0&0\\
0&A_4
\end{pmatrix},~~ \hbox { with } A_4\in Mat(n-r,\mathbb{K}).
$$
Thus
$$
\mathcal{Z}_{J_{n,r}}=\mathbb{K}\begin{pmatrix}
0&0\\
0&A_4
\end{pmatrix}\simeq End(V_2),~~ \hbox { with } A_4\in Mat(n-r,\mathbb{K}),
$$
and
$$
\mathcal{Z}_{J}=P^{-1}\mathcal{Z}_{J_{n,r}}Q^{-1} .
$$
\end{proof}

Now we are able to give the following
\begin{thm}
Let $n$ be an integer with $n\geq 2$.
The Lie algebras
$\left(Mat(n,\K),[~,~]_J\right)$ and $\left(Mat(n,\K),[~,~]_{J'}\right)$ are isomorphic
if and only if the matrices $J$ and $J'$ are equivalent.
\end{thm}
\begin{proof}
Let $J, J'\in Mat(n,\K)$. If $J$ and $J'$ are equivalent then from Proposition
\ref{Equivalent matrices} the
corresponding Lie algebras are isomorphic.

\noindent Now we shall prove the converse.

\noindent Let $J, J'\in Mat(n,\K)$ with $rk(J)=r$ and $rk(J')=s$ with $r\neq s$.

\noindent Let $V$ be a vector space over $\K$ with $\dim V=n$ and put
$V=V_1\oplus V_2=W_1\oplus W_2$ with $\dim V_1=r$ and $\dim W_1=s$.
From Proposition \ref{Equivalent matrices}, we can assume that $J=J_{n,r}$
and $J'=J_{n,s}$ and from Proposition \ref{split}, the reductive part of
$\mathfrak {gl}(n,r,\K)$ is  $End(V_1)\oplus End(V_2)$ while the reductive part of
$\mathfrak {gl}(n,s,\K)$ is $End(W_1)\oplus End(W_2)$. These reductive parts have the same
dimension if and only if $r+s=n$ but in this case they don't have the same dimensionality of their
center. Thus they are not isomorphic and this completes the proof.
\end{proof}

As a consequence of this classification, we have the following

\begin{thm}
Let $\mathfrak H_n$ be the $(2n+1)$-dimensional Heisenberg algebra. Then there
is no faithful finite dimensional linear representation of
$\mathfrak H_n$ on a vector space $V$ with $\dim V\leq n+1$.
\end{thm}

\begin{proof} Suppose that there exists a faithful linear representation
$\rho$ of $\mathfrak H_n$, on a vector space $V$ with $\dim V\leq
n+1$. Let $\rho$ such mapping, then we have $\rho:\mathfrak
H_n\longmapsto \mathfrak{gl}(r,\K)$ which an injective Lie algebra
morphism.

\noindent Put $\mathcal{H}_{n}=Span \{X_i=E_{1,i+1},
Y_i=E_{i+1,n+1}, Z=E_{1,n+2},\quad i=1,\dots, n\}$. We simply
check that $\mathcal{H}_{n}$ is a subalgebra of
$\mathfrak{gl}(n+2,n+1,\K)$, and
$[X_i,Y_j]_{J_{n+1,n}}=\delta_{i,j}Z$, it is the $(2n+1)$-dimensional
Heisenberg algebra. Then we can take $ \mathcal{H}_{n}$ instead of $\mathfrak
H_n$. If such mapping $\rho$ exists, then we must have
$$
\begin{aligned}
\rho\bigl([X_i,Y_i]_{J_{n+2,n+1}}\bigr)&=[\rho(X_i),\rho(Y_i)]\\
\rho(Z)&= [\rho(X_i),\rho(Y_i)],
\end{aligned}
$$ but since $\rho$ is an
injective Lie algebra morphism, then $\rho(Z)=\lambda I_r$ with
$\lambda\neq 0$ wick is impossible because $I_r$ is traceless
matrix.
\end{proof}
\subsection{Example} Let $\mathfrak{H}=span\{Z,Y,X\}$ with the only non vanishing bracket is
$[X,Y]=Z$ then the classical representation of $\mathfrak{H}$ is
$$
X=\begin{pmatrix}
0&1&0\\
0&0&0\\
0&0&0
\end{pmatrix},\quad Y=\begin{pmatrix}
0&0&0\\
0&0&1\\
0&0&0
\end{pmatrix},\quad Z=\begin{pmatrix}
0&0&1\\
0&0&0\\
0&0&0
\end{pmatrix}.
$$
From the last corollary, we can't built a faithful homomorphism
$\mathfrak{H}$ in $\mathfrak{gl}(2,\mathbb{R})$ but we can built
such homomorphism in $\mathfrak{gl}(2,1,\mathbb{R})$ with the
following realization
$$
X=\begin{pmatrix}
0&0\\
1&0
\end{pmatrix},\quad
Y=\begin{pmatrix}
0&1\\
0&0
\end{pmatrix},\quad
Z=\begin{pmatrix}
0&0\\
0&1\\
\end{pmatrix}.
$$

\section { The case of $Mat(n\times m,\K)$} In this section we will
deal with the linear space of strict rectangular matrices
$Mat(n\times m,\K)$ which is not an associative algebra. Fix
$J\in Mat(m\times n, \K)$ and put
$$
[A,B]_J=AJB-BJA,\quad A, B\in Mat(n\times m,\K).
$$
Recall that we have denoted the matrix
$\displaystyle{\begin{pmatrix}
I_r&0_{r,n-r}\\
0_{m-r,r}&0_{m-r,n-r}\end{pmatrix}}$ by $J_{m,n,r}$ and we will denotes the Lie algebra
$\bigl(Mat(n\times m,\K),[~,~]_{J_{m,n,r}}\bigr)$ by $\mathfrak{gl}(n,m,r,\K)$. In the case
when $n=m$, $\mathfrak{gl}(n,n,r,\K)$ is simply $\mathfrak{gl}(n,r,\K)$

\begin{lem} Let $\mathfrak h$ be a Lie sub-algebra of $\mathfrak{gl}(n,\K)$,
then $\mathfrak h$ is a Lie sub-algebra of
$\mathfrak{gl}(p,q,n,\K)$ for any integers $p,q\geq n$.
\end{lem}
\begin{proof}
For any subset $\mathfrak h\subset Mat(n,\K)$ and $q,p\geq n$, we have
$$
[\mathfrak h,\mathfrak h]_{J_{q,p,n}}=[\mathfrak h,\mathfrak h].
$$
And thus the conclusion holds.
\end{proof}
Now let $J\in Mat(m\times n,\K)$, and let's denotes by $\mathcal Z_J$ the
center of the Lie algebra $\left(Mat(n\times m,\K),[~,~]_J\right)$, then we
have

\begin{prop} Let $J\in Mat(m\times n,\K)$ with $rk(J)=r$. Then the center of
$\left(Mat(n\times m,\K),[~,~]_J\right)$ is $(n-r)(m-r)$
dimensional. In particular, if $\mathcal Z_{J_{m,n,r}}$ denotes the center
of $\left(Mat(n\times m,\K),[~,~]_J{_{m,n,r}}\right)$,then $\mathcal Z_{J_{m,n,r}}$ is
spanned by the matrices of the form
$$
A=\begin{pmatrix}
0_r&0_{r,n-r}\\
0_{m-r,r}&A_4\end{pmatrix},\hbox { with } A_4\in
Mat((m-r)\times(n-r),\K).
$$
\end{prop}

\begin{proof}
The proof of this proposition is similar to the proof of the
Proposition \ref{center}.

\end{proof}

Now we can give the following
\begin{thm}Let $J, J'\in Mat(m\times n,\K)$ with $min(m,n)\geq 2$. Then the Lie algebras
$\left(Mat(n\times m,\K),[~,~]_J\right)$ and $\left(Mat(n\times m,\K),[~,~]_{J'}\right)$
are isomorphic if and only if $J$ and $J'$ are equivalent in
$Mat(m\times n,\K).$
\end{thm}

\begin{proof}

Let $J$ and $J'$ two equivalent matrices in $Mat(m\times n, \K)$
then from Proposition \ref{Equivalent matrices}, the corresponding
Lie algebras structures on $Mat(n\times m,\K)$ are isomorphic.

Conversely, let $J, J'$ be in $Mat(m\times n, \K)$ with $
rk(J)\neq rk (J')$, then from Proposition
\ref{center}, we have $
\dim\mathcal{Z}_J\neq \dim\mathcal{Z}_{J'}$, since
$$
\dim\mathcal{Z}_J=(n-rk(J))(m-rk(J))\quad \hbox { and }
\quad \dim\mathcal{Z}_{J'}=(n-rk(J'))(m-rk(J')).
$$

\end{proof}

\begin{cor}
Let $V$ be a vector space over $\K$ with $\dim V\geq 2$. Then $V$ can be equipped with a non trivial
Lie algebra structure.
\end{cor}

\begin{proof}
Let $n$ be the dimension of $V$. We identify $V$ with $Mat(n\times 1,\K)$ and put $J=(1,0,\dots,0)$, then
$(V,[~,~]_J)$ is a non trivial Lie algebra.
\end{proof}
Now we can give an analogue of the Ado's Theorem

\begin{thm}
Every finite-dimensional Lie algebra $\frak g$
over a field $\K$ of characteristic zero can be viewed as a Lie sub-algebra
of $\mathfrak{gl}(n,m,r,\K)$. More precisely
$\frak g$ has a linear representation $\rho$ over $\K$ on $\mathfrak{gl}(n,m,r,\K)$,
that is a faithful representation, making $\frak g$ isomorphic to a subalgebra of $(Hom(\K^n,\K^m),[~,~]_{J_{m,n,r}})$.
\end{thm}
\begin{proof}
By Ado's Theorem, $\mathfrak g$ can be viewed as a
subalgebra of $\mathfrak {gl}(p,\K)$ and the last algebra is also a
subalgebra of $\mathfrak {gl}(n,m,q,\K)$ for any $q\geq p$ and $n,m\geq q$.
\end{proof}

\subsection{Examples}

\noindent (a) Let $V=\R^2$, we identify $V$ with $Mat(2\times 1,\R)$, we choose a basis $(e_1,e_2)$
in $V$ with $e_1=\begin{pmatrix}
  1 \\
  0
\end{pmatrix}, e_2=\begin{pmatrix}
  0 \\
  1
\end{pmatrix}$. Then we can check that $[e_2,e_1]_J=e_1$ with $J=(1,0)$. Then the
two-dimensional affine Lie algebra can be viewed as $(Hom(\R^2,\R), [,~~]_{J_{1,2,1}})$.

\noindent These constructions can be generalized for $V=\R^n$ with $n\geq 2$.

\noindent (b) Let $V=\R^3\simeq Mat(3\times 1,\R)$. Put $e_1=\begin{pmatrix}
    -1 \\
    0 \\
    0 \
  \end{pmatrix}, e_2=\begin{pmatrix}
    0 \\
    1 \\
    0 \
  \end{pmatrix}, e_3=\begin{pmatrix}
    0 \\
    0 \\
    1 \
  \end{pmatrix}$. Then we verify that $[e_1,e_2]_J=e_2,[e_1,e_3]_J=e_3, [e_2,e_3]_J=0$ where $J=(1,0,0)$. It
  is the Lie algebra $\mathfrak g_{3,2}(1)$ (see \cite{B} for notations).

\noindent (c) Let $V=\R^4$, then we can identify $V$ with $Mat(4\times 1,\R)$ and in this case
we take $J=(1,0,0,0)$ or with $Mat(2,\R)$ and in this case we can choose $J=I_2$ or $J=\begin{pmatrix}
  1 & 0 \\
  0 & 0
\end{pmatrix}$.

\section{Contractions and Extensions of Lie algebras}

Let $n, r\in \N$ with $ n\geq r$, $n\geq 2$ and put $J=J_{n,r}$,
then we have the following

\begin{prop}
The Lie algebra $\left(Mat(n,\K),[~,~]_{J_{n,r}}\right)$ is a contraction of
$\mathfrak {gl}(n,\K)$.
\end{prop}

\begin{proof}
Let $(E_{i,j})$ the canonical basis of $Mat(n,\K)$, and define
$$
\frak g=Span\{E'_{i,j}, \quad i, j=1,\dots,n\} ,
$$
with
$$
E'_{i,j}=\left\{\aligned
E_{i,j}& \qquad \hbox { if } i \leq r\hbox { and } j\leq r,\\
\varepsilon E_{i,j}& \qquad \hbox { if } i>r  \hbox { and } j\leq r \hbox { or } j>r \hbox { and } i\leq r,\\
\varepsilon^2E_{i,j}& \qquad \hbox { if } i>r \hbox { and } j>r.
\endaligned\right.
$$
where $ \varepsilon\in \K$. Then we have

$$
[E'_{i,j},E'_{k,l}]= \left\{\aligned
\delta_{j,k}E'_{i,l}-\delta_{l,i}E'_{k,j}&\qquad \hbox { if }
i,j,k,l\leq r,\\
\delta_{j,k}E'_{i,l}&\qquad \hbox { if } i,j,
k\leq r \hbox { and } l>r,\\
-\delta_{l,i}E'_{k,j}            &\qquad \hbox { if } i,j,l\leq
r\quad \hbox {
and } k>r,\\
0&\qquad \hbox { if }
i,j\leq r\quad \hbox { and } k,l>r,\\
0&\qquad \hbox { if }
i,k\leq r\quad \hbox { and } l,j>r,\\
\delta_{j,k}\varepsilon^2E'_{i,l} &\qquad \hbox { if }
i\leq r\quad \hbox { and } j,k,l>r,\\
\delta_{j,k}\varepsilon^2E'_{i,l}-\delta_{l,i}\varepsilon^2E'_{k,j}&\qquad \hbox { if
} i,j,k,l>r.\endaligned\right.
$$
And when $\varepsilon\longrightarrow 0$, then we obtain
$$
\lim_{\varepsilon\rightarrow
0}[E'_{i,j},E'_{k,l}]=[E_{i,j},E_{k,l}]_{J_{n,r}}.
$$

\end{proof}
If we repeat again this contraction, we get a new algebra which can be directly
obtained from the first Lie algebra since $J_{n,r}J_{n,s}=J_{n,t}$ with $t=min(r,s)$.

\subsection{Example} Let $\mathfrak{sl}(2,\R)$ be the semi-simple
(real) Lie algebra spanned by $\{H,X,Y\}$ where
$$
H=\begin{pmatrix}
1&0\\
0&-1
\end{pmatrix},\quad X=\begin{pmatrix}
0&1\\
0&0
\end{pmatrix}
,\quad Y=\begin{pmatrix}
0&0\\
1&0
\end{pmatrix}.
$$
The (usual) brackets are:
$$
[H,X]=2X,\quad [H,Y]=-2Y,\quad [X,Y]=H.
$$
Pick $ J=\begin{pmatrix}
0&0\\
0&1
\end{pmatrix}
$, then we obtain the following brackets:
$$
[H,X]_J=X,\quad [H,Y]_J =-Y \quad [X,Y]_J=0.
$$
The new Lie algebra is completely solvable and thus non isomorphic to $\mathfrak{sl}(2,\R)$
since $\mathfrak{sl}(2,\R)$ is semi-simple.

Let us restrict ourself to the Lie algebra $\mathfrak {gl}(n,\mathbb{K})$.
Then we have the following

\begin{lem}(\cite{Y1}, \cite{Y2}) The map $(A,B)\longmapsto [A,B]_J$ is a two-coboundary for the
adjoint representation of $\mathfrak {gl}(n,\K)$.
\end{lem}

\begin{proof}
We can easily check that
$$
[A,B]_J=ad_A\alpha(B)-ad_B\alpha(A)-\alpha([A,B])=(d\alpha)(A,B).
$$
where
$$
\alpha(X)=\frac{1}{2}(XJ+JX).
$$
\end{proof}
Let $t\in[0,1]$ and $J=J_{n,r}$ (with $r<n$), then we can write
$$
[A,B]_{(1-t)I+tJ}=[A,B]+t[A,B]_{J-I},
$$
thus since $[A,B]_{J-I}$ is a coboundary, $t\longmapsto[A,B]_{(1-t)I+tJ}$ is a deformation infinitesimally trivial at
$t=0$ from $\mathfrak {gl}(n,\K)$ to $\mathfrak{gl}(n,r,\K)$. For any $t<1$, the Lie algebra
$\left(Mat(n,\K),[~,~]_{(1-t)I+tJ}\right)$ is isomorphic to $\mathfrak {gl}(n,\K)$ but for $t=1$ we get a completely different structure which is labelled by the bracket $[~,~]_J$.
The mapping
$$
\Psi_t:\left(Mat(n,\K),[~,~]_{(1-t)I+tJ}\right)\longrightarrow \mathfrak {gl}(n,\K),\quad
\left(
  \begin{array}{cc}
   X_1& X_2 \\
   X_3& X_4 \\
  \end{array}
\right)\longmapsto\left(
  \begin{array}{cc}
   X_1& (1-t)X_2 \\
   X_3& (1-t)X_4 \\
  \end{array}
\right)
$$
is for any $t\in [0,1[$  invertible and $
[X,Y]_{(1-t)I+tJ}=\displaystyle{\Psi_t^{-1}\left([\Psi_t(X),\Psi_t(Y)]\right)},
$
and we have $\frac{\partial}{\partial t}([~,~]_{(1-t)I+tJ})_{| t=0 }$ is the 2-coboundary:
$(A,B)\mapsto [A,B]_{J-I}.$
\vskip 0.5cm
\begin{lem}
The vector space $Mat(n,\K)$ is a $\left(Mat(n,\K),[~,~]_J\right)$-module.
\end{lem}

\begin{proof}
We can verify that the map
$$
\aligned
\varphi:\left(Mat(n,\K),[~,~]_J\right)&\longrightarrow \mathfrak {gl}\left((Mat(n,\K)\right)\\
                        A&\longmapsto ad_A^J,
\endaligned
$$
where
$$
ad_A^J(B)=[A,B]_J,
$$
is a Lie algebra homomorphism since $[~,~]_J$ is a Lie bracket on
$Mat(n,\K)$.
\end{proof}

Let $\mathfrak{g}$ be a Lie algebra over $\mathbb{K}$ and
$(X_1,\dots,X_r)$ with the following commutations roles
$[X_i,X_j]=\displaystyle{\sum_kc_{i,j}^kX_k}.$ By the Ado's
theorem, there is a faithful representation $\rho$ of $\mathfrak
g$ in $\mathfrak{gl}(n,\K)$ for some integer $n$. Let us identify
$\mathfrak g$ with its image under $\rho$. Then we have the
following

\begin{prop}
Let $\frak{h}= Mat(r\times (n-r),\K)$, $\frak{h}'= Mat((n-r)\times
r,\K)$ and  $\frak{h}"= Mat(n-r,\K)$. Then we have the following
subalgebras inclusions
$$
\begin{pmatrix}
\mathfrak g&\mathfrak{h}\\
0&0
\end{pmatrix}\subset \begin{pmatrix}
\mathfrak g&\mathfrak{h}\\
0&\mathfrak{h}"
\end{pmatrix}\subset
\mathfrak{gl}(n,r,\mathbb{K})
$$
and
$$
\begin{pmatrix}
\mathfrak g&0\\
\mathfrak{h}'&0
\end{pmatrix}\subset
\begin{pmatrix}
\mathfrak g&0\\
\mathfrak{h}'&\mathfrak{h}"
\end{pmatrix}\subset\mathfrak{gl}(n,r,\mathbb{K}).
$$

\end{prop}

\end{document}